\documentclass[a4paper]{article}

\usepackage[latin1]{inputenc}
\usepackage[T1]{fontenc}
\usepackage{ae,aecompl}
\usepackage{amsmath,amsfonts,amsthm}
\usepackage[tableposition=top,font=small,labelfont=bf,format=hang]{caption}
\usepackage{subfig}

\usepackage{booktabs}
\usepackage[sort&compress,numbers]{natbib}
\usepackage[dvips]{graphicx}
\usepackage[dvips,colorlinks=false]{hyperref}
\hypersetup{
pdfauthor = {Lars Eirik Danielsen, Matthew G. Parker, Constanza Riera, and Joakim Grahl Knudsen},
pdftitle = {On Graphs and Codes Preserved by Edge Local Complementation},
pdfkeywords = {graphs, isomorphism, edge local complementation, pivot, binary linear codes}
}

\newtheorem{thm}{Theorem}
\newtheorem{prop}{Proposition}

\newtheorem{lem}{Lemma}
\theoremstyle{definition}
\newtheorem{defn}{Definition}

\theoremstyle{remark}

\newtheorem{prob}{Problem}

\title{On Graphs and Codes Preserved by \\Edge Local Complementation}

\author{Lars Eirik Danielsen\thanks{Department of Informatics, University of Bergen, PO Box 7803, \mbox{N-5020} Bergen, Norway. \texttt{\{\href{mailto:larsed@ii.uib.no}{larsed},\href{mailto:matthew@ii.uib.no}{matthew},\href{mailto:joakimk@ii.uib.no}{joakimk}\}@ii.uib.no}\hfill \texttt{http://www.ii.uib.no/\~{}\{\href{http://www.ii.uib.no/\~larsed}{larsed},\href{http://www.ii.uib.no/\~matthew}{matthew},\href{http://www.ii.uib.no/\~joakimk}{joakimk}\}}} 
\and Matthew G. Parker\footnotemark[1]
\and Constanza Riera\thanks{Bergen University College, PO Box 7030, N-5020 Bergen, Norway. \texttt{\href{mailto:csr@hib.no}{csr@hib.no}}}
\and Joakim Grahl Knudsen\footnotemark[1]
}

\date{August 8, 2013}

\begin{document}

\maketitle

\begin{abstract}
  Orbits of graphs under local complementation (LC) and edge local
  complementation (ELC) have been studied in several different
  contexts. For instance, there are connections between orbits of
  graphs and error-correcting codes. We define a new graph class,
  ELC-preserved graphs, comprising all graphs that have an ELC orbit
  of size one. Through an exhaustive search, we find all ELC-preserved
  graphs of order up to 12 and all ELC-preserved bipartite graphs of
  order up to 16. We provide general recursive constructions for
  infinite families of ELC-preserved graphs, and show that all known
  ELC-preserved graphs arise from these constructions or can be
  obtained from Hamming codes. We also prove that certain pairs of
  ELC-preserved graphs are LC equivalent. We define ELC-preserved
  codes as binary linear codes corresponding to bipartite
  ELC-preserved graphs, and study the parameters of such codes.
\end{abstract}

\section{Introduction}\label{sec:intro}

The \emph{local complementation} (LC) operation was first defined by
Kotzig~\cite{kotzig} and later studied by de Fraysseix~\cite{hubert},
Fon-der-Flaas~\cite{flaas}, and Bouchet~\cite{bouchet}. Bouchet also
introduced \emph{edge local complementation} (ELC)~\cite{bouchet}, an
operation which is also known as \emph{pivoting} on a graph.  LC
orbits of graphs have been used to study \emph{quantum graph
  states}~\cite{hein,nest2}, which are equivalent to \emph{self-dual
  additive codes over $\mathbb{F}_4$}~\cite{calderbank}.  LC orbits
have been used to classify such codes~\cite{classlc}.  There are also
connections between graph orbits and properties of \emph{Boolean
  functions}~\cite{pivotboolean, genbent}.  \emph{Interlace
  polynomials} of graphs have been defined with respect to both
LC~\cite{aigner} and ELC~\cite{interlace}. These polynomials encode
certain properties of the graph orbits, and were originally used to
study a problem related to DNA sequencing~\cite{eulerdna}.
Connections between interlace polynomials and error-correcting codes
have also been studied~\cite{interlacecodes}.
Bouchet~\cite{bouchetcircle} proved that a graph is a \emph{circle
  graph} if and only if certain induced subgraphs, or obstructions, do not
appear anywhere in its LC orbit.  Similarly, circle graph obstructions
under ELC were described by Geelen and Oum~\cite{circlepivot}.

In this paper, we introduce \emph{ELC-preserved graphs} as a new class
of graphs, namely those that are invariant under the ELC operation and
therefore having trivial ELC orbits of size one. In light of the
previous works and various applications listed in the previous
paragraph, we feel that ELC-preserved graphs are fundamental objects
worthy of study and, for this paper, we consider both graph- and
code-theoretic interpretations of these objects.

Bipartite graphs correspond to binary linear error-correcting
codes. ELC can be used to generate orbits of equivalent codes and has
previously been used to classify codes~\cite{classelc}.  It has also
been shown that ELC can improve the performance of iterative
decoding~\cite{castle,isit,wbelc,isit2010}. ELC-preserved graphs are
of particularly interest in this context, since for such graphs the
decoding algorithm is equivalent to a variant of \emph{permutation
  decoding}~\cite{castle,halford}.

We will show that the class of codes corresponding to bipartite
ELC-preserved graphs, which we will call \emph{ELC-preserved codes},
is a superset of both the Hamming codes and the extended Hamming
codes, which makes it an interesting class of codes. When it comes to
practical applications in iterative decoding, we conclude that the
ELC-preserved criterion might be too strict to obtain good codes with
appropriate length. However, we suggest that ELC-preserved graphs
could be a building block for good codes, and when we look at ``almost
ELC-preserved'' graphs with ELC-orbits of size two, we find both the
Golay code and a BCH code. Other relaxations of the ELC-preserved
criterion yielding practical error-correction applications have been
considered in other works~\cite{castle,wbelc}. In this paper we focus
on the theoretical properties of ELC-preserved graphs and codes.

This paper is organized as follows. Section~\ref{sec:prelim}
introduces all necessary notation from graph theory and coding theory.
In Section~\ref{sec:enum}, we show that there do exist non-trivial
bipartite and non-bipartite ELC-preserved graphs. We find all
ELC-preserved graphs of order up to 12 and all ELC-preserved bipartite
graphs of order up to 16. In Section~\ref{sec:const}, we show that
\emph{star graphs} and \emph{complete graphs} as well as graphs
corresponding to \emph{Hamming codes} and \emph{extended Hamming
  codes} are ELC-preserved. We then prove that more ELC-preserved
graphs can be obtained from four recursive constructions. Given a
bipartite ELC-preserved graph, a larger bipartite ELC-preserved graph
is constructed by \emph{star expansion}. Similarly, \emph{clique
  expansion} produces non-bipartite ELC-preserved
graphs. \emph{Hamming expansion} and the related \emph{Hamming clique
  expansion} use a special graph of order seven, corresponding to a
Hamming code, to obtain new ELC-preserved graphs. In
Section~\ref{sec:class}, we show that all ELC-preserved graphs of
order up to 12, and all ELC-preserved bipartite graphs of order up to
16, are obtained from these constructions.  We also prove that certain
pairs of ELC-preserved graphs are LC equivalent.  In particular, from
extended Hamming codes, we obtain new non-bipartite ELC-preserved
graphs via LC.  The properties of ELC-preserved codes obtained from
star expansion and Hamming expansion are described in
Section~\ref{sec:coding}.  In particular, we enumerate and construct
new self-dual ELC-preserved codes.  In Section~\ref{sec:sizetwo} we
briefly consider the generalization from ELC-preserved graphs to
graphs with orbits of size two, and study the corresponding codes.
Finally, in Section~\ref{sec:conc}, we conclude with some ideas for
future research.

\section{Preliminaries}\label{sec:prelim}

\subsection{Graphs}

A \emph{graph} is a pair $G=(V,E)$ where $V$ is a set of
\emph{vertices}, and $E \subseteq V \times V$ is a set of
\emph{edges}.  The \emph{order} of $G$ is $n = |V|$.  A graph of order
$n$ can be represented by an $n \times n$ \emph{adjacency matrix}
$\Gamma$, where $\Gamma_{i,j} = 1$ if $\{i,j\} \in E$, and
$\Gamma_{i,j} = 0$ otherwise.  We will only consider \emph{simple}
\emph{undirected} graphs, whose adjacency matrices are symmetric with
all diagonal elements being 0, i.e., all edges are bidirectional and
no vertex can be adjacent to itself.  The \emph{neighborhood} of $v
\in V$, denoted $N_v \subset V$, is the set of vertices connected to
$v$ by an edge.  The number of vertices adjacent to $v$ is called the
\emph{degree} of $v$.  The \emph{induced subgraph} of $G$ on $W
\subseteq V$ is the graph that has $W$ as a set of vertices and
has all edges in $E$ whose endpoints are both in $W$. The
\emph{complement} of $G$ is a graph with the same vertex set, $V$, but
whose edge set consists of the edges not present in G, i.e., the
complement of $E$. (Note that the complement will also be a simple
graph, i.e., no loops are introduced.)  Two graphs $G=(V,E)$ and
$G'=(V,E')$ are \emph{isomorphic} if and only if there exists a
permutation $\pi$ on $V$ such that $\{u,v\} \in E$ if and only if
$\{\pi(u), \pi(v)\} \in E'$.  A \emph{path} is a sequence of distinct vertices,
$(v_1,v_2,\ldots,v_i)$, such that $\{v_1,v_2\}, \{v_2,v_3\},$ $\ldots,
\{v_{i-1},v_{i}\} \in E$.  A graph is \emph{connected} if there is a
path from any vertex to any other vertex in the graph.  A graph is
\emph{bipartite} if its set of vertices can be decomposed into two
disjoint sets, called partitions, such that no two vertices within the
same set are adjacent, and non-bipartite otherwise.  We call a graph
\emph{$(a,b)$-bipartite} if these partitions are of size $a$ and $b$,
respectively.

\begin{defn}[\hspace{1pt}\hspace{-1pt}\cite{flaas,bouchet,hubert}]
  Given a graph $G=(V,E)$ and a vertex $v \in V$, let $N_v \subset V$
  be the neighborhood of $v$.  \emph{Local complementation} (LC) on
  $v$ transforms $G$ into $G * v$ by replacing the induced subgraph of
  $G$ on $N_v$ by its complement. (For an example, see
  Fig.~\ref{fig:lcexample})
\end{defn}

\begin{figure}
 \centering
 \subfloat[The graph $G$]
 {\hspace{5pt}\includegraphics[width=.25\linewidth]{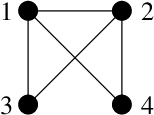}\hspace{5pt}\label{fig:lcexample1}}
 \quad
 \subfloat[The graph $G*1$]
 {\hspace{5pt}\includegraphics[width=.25\linewidth]{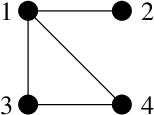}\hspace{5pt}\label{fig:lcexample2}}
 \caption{Example of local complementation}\label{fig:lcexample}
\end{figure}

\begin{defn}[\hspace{1pt}\hspace{-1pt}\cite{bouchet}]\label{prop:triplelc}
  Given a graph $G=(V,E)$ and an edge $\{u,v\} \in E$, \emph{edge
    local complementation} (ELC) on $\{u,v\}$ transforms $G$ into
  $G^{(u,v)} = G*u*v*u = G*v*u*v$.
\end{defn}

\begin{defn}[\hspace{1pt}\hspace{-1pt}\cite{bouchet}]\label{def:elc}
  ELC on $\{u,v\}$ can equivalently be defined as follows. Decompose
  $V\setminus \{u,v\}$ into the following four disjoint sets, as
  visualized in Fig.~\ref{fig:elc}.
  \renewcommand{\labelenumi}{$\Alph{enumi}$}
  \begin{enumerate}
  \item Vertices adjacent to $u$, but not to $v$.
  \item Vertices adjacent to $v$, but not to $u$.
  \item Vertices adjacent to both $u$ and $v$.
  \item Vertices adjacent to neither $u$ nor $v$.
  \end{enumerate}
  To obtain $G^{(u,v)}$, perform the following procedure.  For any
  pair of vertices $\{x,y\}$, where $x$ belongs to class $A$, $B$, or
  $C$, and $y$ belongs to a different class $A$, $B$, or $C$,
  ``toggle'' the pair $\{x, y\}$, i.e., if $\{x,y\} \in E$, delete the
  edge, and if $\{x,y\} \not\in E$, add the edge $\{x,y\}$ to
  $E$. Finally, swap the labels of vertices $u$ and $v$.
\end{defn}

\begin{figure}
 \centering
 \includegraphics[width=.40\linewidth]{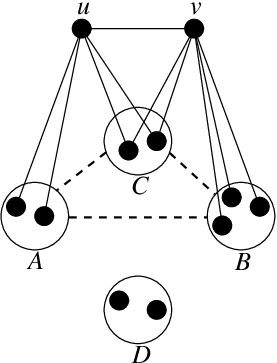}
 \caption{Visualization of the ELC operation}\label{fig:elc}
\end{figure}

\begin{defn}
  The graphs $G$ and $G'$ are \emph{LC-equivalent} (resp.
  \emph{ELC-equivalent}) if a graph isomorphic to $G'$ can be obtained
  by applying a finite sequence of LC (resp. ELC) operations to $G$.
  The \emph{LC orbit} (resp. \emph{ELC orbit}) of $G$ is the set of
  all non-isomorphic graphs that can be obtained by performing any
  finite sequence of LC (resp. ELC) operations on~$G$.
\end{defn}

For bipartite graphs, we can simplify the ELC operation, since the set
$C$ in Definition~\ref{def:elc} must be empty. Given a bipartite graph
$G=(V,E)$ and an edge $\{u,v\} \in E$, $G^{(u,v)}$ can be obtained by
``toggling'' all edges between the sets $N_u \setminus \{v\}$ and $N_v
\setminus \{u\}$, followed by a swapping of vertices $u$ and $v$.
Moreover, if $G$ is an $(a,b)$-bipartite graph, then, for any edge
$\{u,v\} \in E$, $G^{(u,v)}$ must also be
$(a,b)$-bipartite~\cite{pivotboolean}. Note that LC does not, in
general, preserve bipartiteness.  It follows from
Definition~\ref{prop:triplelc} that every LC orbit can be partitioned
into one or more ELC orbits.  If $G=(V,E)$ is a connected graph, then,
for any vertex $v \in V$, $G*v$ must also be connected.  Likewise, for
any edge $\{u,v\} \in E$, $G^{(u,v)}$ must be connected.

\begin{defn}
  A graph $G$ is \emph{ELC-preserved} if for any edge
  $\{u,v\} \in E$, $G^{(u,v)}$ is isomorphic to $G$. In other words,
  $G$ is ELC-preserved if and only if the ELC orbit has $G$ as the only element.
\end{defn}

We only consider connected graphs, since a disconnected graph is
ELC-preserved if and only if its connected components are
ELC-preserved. Trivially, empty graphs, i.e., graphs with no edges,
are ELC-preserved. 

\subsection{Codes}

A binary linear code, $\mathcal{C}$, is a linear subspace of
$\mathbb{F}_2^n$ of dimension $k$. The $2^k$ elements of $\mathcal{C}$
are called \emph{codewords}.  The \emph{Hamming weight} of a codeword
is the number of non-zero components. The \emph{minimum distance} of
$\mathcal{C}$ is equal to the smallest non-zero weight of any codeword
in $\mathcal{C}$.  A code with minimum distance~$d$ is called an
$[n,k,d]$ code.
Two codes are \emph{equivalent} if one can be obtained from the other
by a permutation of the coordinates.  A permutation that maps a code
to itself is called an \emph{automorphism}. All automorphisms of
$\mathcal{C}$ make up its \emph{automorphism group}. We define the
\emph{dual code} of $\mathcal{C}$ with respect to the standard inner
product, $\mathcal{C}^\perp = \{ \boldsymbol{u} \in \mathbb{F}_2^n
\mid \boldsymbol{u} \cdot \boldsymbol{c}=0 \text{ for all }
\boldsymbol{c} \in \mathcal{C} \}$.  $\mathcal{C}$~is called
\emph{self-dual} if $\mathcal{C} = \mathcal{C}^\perp$, and
\emph{isodual} if $\mathcal{C}$ is equivalent to $\mathcal{C}^\perp$.
The code $\mathcal{C}$ can be defined by a $k \times n$
\emph{generator matrix}, $C$, whose rows span $\mathcal{C}$.  By
column permutations and elementary row operations $C$ can be
transformed into a matrix of the form $C' = (I \mid P)$, where $I$ is
a $k \times k$ identity matrix, and $P$ is some $k \times (n-k)$
matrix. The matrix $C'$, which is said to be of \emph{standard form},
generates a code which is equivalent to $\mathcal{C}$.  The matrix $H'
= (P^\text{T} \mid I)$, where $I$ is an $(n-k) \times (n-k)$ identity
matrix is the generator matrix of $\mathcal{C}'^\perp$ and is called
the \emph{parity check matrix} of $\mathcal{C}'$.

\begin{defn}[\hspace{1pt}\hspace{-1pt}\cite{curtis,rijmen}]\label{def:code}
  Let $\mathcal{C}$ be a binary linear $[n,k]$ code with generator
  matrix $C = (I \mid P)$. Then the code $\mathcal{C}$ corresponds to
  the $(k,n-k)$-bipartite graph on $n$ vertices with adjacency matrix
\[
\Gamma = \begin{pmatrix}\boldsymbol{0}_{k\times k} & P \\ 
P^\text{T} & \boldsymbol{0}_{(n-k)\times (n-k)}\end{pmatrix},
\]
where $\boldsymbol{0}$ denotes all-zero matrices of the specified
dimensions.
\end{defn}

\begin{thm}[\hspace{1pt}\hspace{-1pt}\cite{classelc}]\label{thm:elccode}
  Applying any sequence of ELC operations to a graph corresponding to
  a code $\mathcal{C}$ will produce another graph corresponding to the
  code $\mathcal{C}$.  Moreover, graphs corresponding to
  equivalent codes will always belong to the same ELC orbit (up to isomorphism).
\end{thm}

Note that, up to isomorphism, one bipartite graph corresponds to both
the code $\mathcal{C}$ generated by $(I \mid P)$, and the code
$\mathcal{C}^\perp$ generated by $(P^\text{T} \mid I)$. When $\mathcal{C}$ is
isodual, the ELC-orbit of the associated graph corresponds to a single
equivalence class of codes. Otherwise, the ELC-orbit corresponds to two
equivalence classes, that of $\mathcal{C}$ and that of
$\mathcal{C}^\perp$~\cite{classelc}.

\begin{defn}
  An \emph{ELC-preserved code} is a binary linear code corresponding
  to an ELC-preserved bipartite graph.
\end{defn}

It follows from Theorem~\ref{thm:elccode} that ELC allows us to jump
between all standard form generator matrices of a code. Hence an
ELC-preserved code is a code that has only one standard form generator
matrix, up to column permutations.

\begin{thm}[\hspace{1pt}\hspace{-1pt}\cite{classelc}]\label{thm:distance}
  The minimum distance of an $[n,k,d]$ binary linear code
  $\mathcal{C}$ is $d = \delta + 1$, where $\delta$ is the smallest
  vertex degree of any vertex in a fixed partition of size $k$ over
  all graphs in the associated ELC orbit. The minimum vertex degree in
  the other partition over the ELC orbit gives the minimum distance of
  $\mathcal{C}^\perp$.
\end{thm}

For an ELC-preserved graph, Theorem~\ref{thm:distance} means that the
minimum distance of the associated code, and its dual code, can be
found simply by finding the minimum vertex degree in each partition of
the graph.

In the technique of iterative decoding with
ELC~\cite{castle,isit,wbelc,isit2010}, labeled graphs are used, so
that ELC is equivalent to row additions on an initial generator matrix
of the form $(I \mid P)$, which means that the corresponding code is
preserved. (It is the parity check matrix of the code that is actually
used for decoding, but we have already seen that, up to isomorphism,
the bipartite graph corresponding to the generator matrix and parity
check matrix of a code is the same.)  For an ELC-preserved code, all
generator matrices must be column permutations of one unique generator
matrix, and hence these permutations must all be automorphisms of the
code. It follows that iterative decoding with ELC on an ELC-preserved
code is equivalent to a variant of \emph{permutation
  decoding}~\cite{castle,halford}.

\section{Enumeration}\label{sec:enum}

From previous classifications~\cite{classlc,classelc}, we know the ELC
orbit size for all graphs of order $n \le 12$, and all bipartite
graphs of order $n \le 15$. (A database of ELC orbits is available
on-line at \url{http://www.ii.uib.no/~larsed/pivot/}.) We find that a
small number of ELC orbits of size one exist for each
order~$n$. Despite the much smaller number of bipartite graphs, there
are approximately the same number of ELC-preserved bipartite and
non-bipartite graphs for $n \le 12$.  The numbers of ELC-preserved
graphs, together with the total numbers of ELC orbits, are given in
Table~\ref{tab:enumerate}. Note that all numbers are for connected
graphs.

By using an extension technique we were also able to generate all
ELC-preserved bipartite graphs of order $n=16$. Given the 1,156,716
ELC orbit representatives for $n=15$, we extend each $(a,b)$-bipartite
graph in $2^a+2^b-2$ ways, by adding a new vertex and connecting it to
all possible combinations of at least one of the old vertices. The
complete set of extended graphs is significantly smaller than that set
of all bipartite connected graphs of order 16, but it must contain at
least one representative from each ELC orbit. To see that this is
true, consider a connected bipartite graph $G$ of order 16.  The
induced subgraph on any 15 vertices of $G$ must be ELC-equivalent to
one of the graphs that were extended to form the extended set, and
hence there must be at least one graph in the extended set that is
ELC-equivalent to $G$.  We check each member of the extended set, and
find that there are 6 connected bipartite ELC-preserved graphs of
order 16.  Note that this is the same extension technique that was
used to classify ELC orbits~\cite{classelc}, but checking if a graph
is ELC-preserved is much faster than generating its entire ELC orbit,
since we only need to consider ELC on each edge of the graph, and can
stop and reject the graph as soon as a second orbit member is
discovered.

\begin{table}
\centering
\caption{Number of non-bipartite ELC orbits ($nb_n$), non-bipartite ELC-preserved graphs ($nbp_n$),
bipartite ELC orbits ($b_n$), and bipartite ELC-preserved graphs ($bp_n$)}
\label{tab:enumerate}
\begin{tabular}{rrrrr}
\toprule
$n$ & $nb_n$ & $nbp_n$ & $b_n$ & $bp_n$ \\
\midrule
2  &           - &        - &          1 &          1 \\
3  &           1 &        1 &          1 &          1 \\
4  &           2 &        1 &          2 &          1 \\
5  &           7 &        1 &          3 &          1 \\
6  &          27 &        2 &          8 &          2 \\
7  &         119 &        1 &         15 &          2 \\
8  &         734 &        2 &         43 &          3 \\
9  &       6,592 &        3 &        110 &          2 \\
10 &     104,455 &        3 &        370 &          2 \\
11 &   3,369,057 &        2 &      1,260 &          1 \\
12 & 231,551,924 &        6 &      5,366 &          5 \\
13 &             &          &     25,684 &          1 \\
14 &             &          &    154,104 &          5 \\
15 &             &          &  1,156,716 &          4 \\
16 &             &          &          ? &          6 \\
\bottomrule
\end{tabular}
\end{table}

\section{Constructions}\label{sec:const}

For all $n \ge 2$, there is a bipartite ELC-preserved graph of order
$n$, namely the \emph{star graph}, denoted $s^n$. This graph has one
vertex, $v$, of degree $n-1$ and $n-1$ vertices, $u_1, u_2, \ldots,
u_{n-1}$, of degree 1. Clearly the graph is ELC-preserved, since for
all edges $\{u_i, v\}$, $N_{u_i} \setminus \{v\} = \emptyset$.  The
construction given in Theorem~\ref{sexapand} gives us more bipartite
ELC-preserved graphs.  For brevity, we will denote $N_v^u = N_v
\setminus (N_u \cup \{u\})$.  Let $e^n$ denote the \emph{empty graph}
on $n$ vertices, i.e., a graph with no edges.

\begin{defn}[\hspace{1pt}\hspace{-1pt}\cite{bollobas, interlace}]\label{def:subst}
  Given a graph $G=(V,E)$, a vertex $v \in V$, and another graph
  $H=(V',E')$, where $V \cap V' = \emptyset$, by \emph{substituting}
  $v$ with $H$, we obtain the graph $G'=((V \setminus \{v\}) \cup V',
  E'')$, where $E''$ is obtained by taking the union of $E$ and $E'$,
  removing all edges incident on $v$, and joining all vertices in $V'$
  to $w$ whenever $\{v,w\} \in E$.
\end{defn}

\begin{defn}[\hspace{1pt}\hspace{-1pt}\cite{monaghan}]
  Given a graph $G=(V,E)$, and a vertex $v \in V$, we add a
  \emph{pendant} at $v$ by adding a new vertex $w$ to $V$ and a new
  edge $\{v, w\}$ to $E$.
\end{defn}

\begin{thm}[Star expansion]\label{sexapand}
  Given an ELC-preserved bipartite graph $G=(V,E)$ on $k$ vertices and
  an integer $m > 1$, we obtain an ELC-preserved bipartite graph
  $S^m(G)$ on $n=km$ vertices by substituting all vertices in one
  partition of $G$ with $e^m$ and adding $m-1$ pendants to all
  vertices in the other partition.
\end{thm}
\begin{proof}
  Let $\{u,v\} \in E$. Without loss of generality, assume that $u$ is
  substituted by $u_1, \ldots, u_m$, all incident on $v$. Moreover,
  pendant vertices $w_1, \ldots, w_{m-1}$ are added, with $v$ as their
  only neighbor.  Clearly ELC on $\{v, w_i\}$ is ELC-preserving. Due
  to symmetries, it only remains to show that ELC on an edge $\{u_i,
  v\}$ preserves $S^m(G)$.  In the graph $G$, let $A = N_u^v$ and $B =
  N_v^u$.  In the graph $S^m(G)$, $N_{u_i}^v = A$, and $N_v^{u_i} =
  (B_1 \cup \cdots \cup B_m) \cup C \cup D$, where $C = \{w_1, \ldots,
  w_{m-1}\}$ and $D = \{u_1, \ldots, u_m\} \setminus \{u_i\}$.  The
  subgraph induced on $A \cup B_j$ in $S^m(G)$, for $1 \le j \le m$,
  is isomorphic to the subgraph induced on $A \cup B$ in $G$.  ELC on
  $\{u_i, v\}$ means that we toggle all pairs of vertices between
  $N_{u_i}^v$ and $N_v^{u_i}$.  Toggling pairs between $A$ and $B_j$,
  for $1 \le j \le m$, preserves $S^m(G)$, since toggling pairs
  between $A$ and $B$ preserves $G$.  (The fact that all vertices in
  $A$ have $m-1$ added pendants has no effect on this.)  Finally, in
  addition to swapping $u_i$ and $v$, ELC has the effect of toggling
  pairs of vertices between $A$ and $C$, and between $A$ and $D$.  In
  $S^m(G)$, all vertices in $A$ are connected to all vertices in $D$,
  and no vertex in $A$ is connected to any vertex in $C$. The sets $C$
  and $D$ are both of size $m-1$, the vertices in $C$ have no other
  neighbors than $v$, and the vertices in $D$ have no other neighbors
  than $A \cup \{v\}$.  Hence ELC on $\{u_i, v\}$ simply swaps the
  vertices in $C$ with the vertices in $D$.  This means that
  $S^m(G)^{(u_i,v)}$ is isomorphic to $S^m(G)$, and it follows that
  $S^m(G)$ is ELC-preserved.  Furthermore, $S^m(G)$ must be bipartite,
  since substituting vertices by empty graphs and adding pendants
  cannot make a bipartite graph non-bipartite.
\end{proof}

\begin{figure}
 \centering
 {\hfill
 \subfloat[The graph $S_-^2(s^3)$]{\includegraphics[width=.4\linewidth]{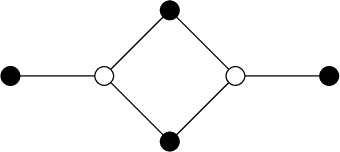}\label{fig:starexample1}}
 \hfill
 \subfloat[The graph $S_-^2(S_-^2(s^3))$]{\includegraphics[width=.55\linewidth]{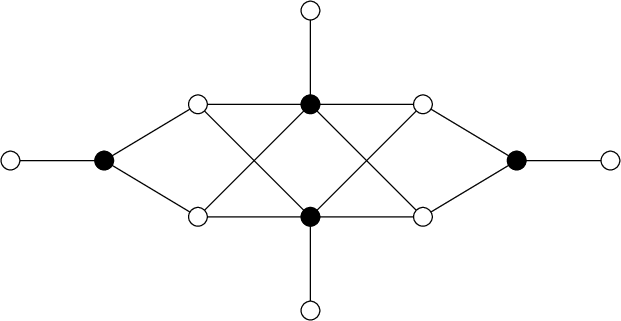}\label{fig:starexample2}}
 \hfill}
 \caption{Examples of star expansion}\label{fig:starexample}
\end{figure}

Examples of graphs obtained by star expansion are shown in
Fig.~\ref{fig:starexample}. From Theorem~\ref{sexapand} we can obtain
two different graphs, by choosing in which partition of $G$ we
substitute vertices by $e^m$.  In our examples, when the partitions of
$G$ are of unequal size, we write $S_+^m(G)$ when we substitute the
vertices in the largest partition, and $S_-^m(G)$ when we substitute
the vertices in the smallest partition. In the cases where the
partitions are of equal size, $S^m(G)$ will give the same graph for
both partitions in all examples in this paper.  If $G$ is an $(r,
k-r)$-bipartite graph, then $S^m(G)$ will be $(r+k(m-1),
k-r)$-bipartite. Since its output is always bipartite, the star
expansion construction can be iterated to obtain new ELC-preserved
graphs, such as the graph $S_-^2(S_-^2(s^3))$ of order 12, shown in
Fig.~\ref{fig:starexample2}. However, some of these iterated
constructions can be simplified. For instance, it is easy to verify
that $S_+^m(s^k) = s^{km}$ and $S_+^{m_2}(S_-^{m_1}(s^k)) =
S_-^{m_1m_2}(s^k)$.

For all $n \ge 3$, there is a non-bipartite ELC-preserved graph on $n$
vertices, namely the \emph{complete graph}, denoted $c^n$. This graph
has $n$ vertices, $v_1, v_2, \ldots, v_n$, of degree $n-1$.  Clearly
the graph is ELC-preserved, since for all edges $\{v_i, v_j\}$,
$N_{v_i} = N_{v_j}$, and hence the sets $A$ and $B$ in
Fig.~\ref{fig:elc} are empty.  The following more general construction
gives us more non-bipartite ELC-preserved graphs.

\begin{thm}[Clique expansion]\label{cexapand}
  Given an ELC-preserved graph $G$ on $k$ vertices and an integer $m >
  1$, we obtain an ELC-preserved non-bipartite graph $C^m(G)$ on
  $n=km$ vertices by substituting all vertices of $G$ with $c^m$.
\end{thm}
\begin{proof}
  Let $\{u,v\} \in E$. Let $u$ be substituted by $u_1, \ldots, u_m$,
  and let $v$ be substituted by $v_1, \ldots, v_m$.  ELC on any edge
  within a substituted subgraph, such as $\{u_i, u_j\}$, must preserve
  $C^m(G)$, since $N_{u_i} = N_{u_j}$.  Due to symmetries, it only
  remains to show that ELC on an edge $\{u_i, v_j\}$ preserves
  $C^m(G)$.  In the graph $G$, let $A = N_u^v$, $B = N_v^u$, and $C =
  N_u \cap N_v$.  In the graph $C^m(G)$, $N_{u_i}^{v_j} = A_1 \cup
  \cdots \cup A_m$, $N_{v_j}^{u_i} = B_1 \cup \cdots \cup B_m$, and
  $N_{u_i} \cap N_{v_j} = (C_1 \cup \cdots \cup C_m) \cup U \cup V$,
  where $U = \{u_1, \ldots, u_m\} \setminus \{u_i\}$ and $V = \{v_1,
  \ldots, v_m\} \setminus \{v_j\}$.  Let $X,Y \in \{A,B,C\}$, $X \ne
  Y$. All subgraphs in $C^m(G)$ induced on $X_r$ are isomorphic to
  subgraphs in $G$ induced on $X$.  A vertex $x_r \in X_r$ is
  connected to a vertex $y_s \in Y_s$ in $C^m(G)$, for $1 \le r,s \le
  m$ if and only if $x \in X$ is connected to $y \in Y$ in $G$.
  Hence, toggling pairs between $X_r$ and $Y_s$, for $1 \le r,s \le
  m$, preserves $C^m(G)$ since toggling pairs between $X$ and $Y$
  preserves $G$.  (The fact that edges have been added between $X_r$
  and $X_t$, for $1 \le r,t \le m$, by the clique substitution, has no
  effect on this, since the subgraphs in $C^m(G)$ induced on $X_r \cup
  X_t$ are isomorphic for all $1 \le r,t \le m$.)  The final effect of
  ELC on $\{u_i, v_j\}$ is to toggle all pairs between $U \cup V$ and
  $A_1 \cup \cdots \cup A_m$, and all pairs between $U \cup V$ and
  $B_1 \cup \cdots \cup B_m$.  But, since we also swap $u_i$ and
  $v_j$, the total effect is equivalent to swapping $u_r$ and $v_r$
  for all $1 \le r \le m$. It follows that $C^m(G)^{(u_i,v_j)}$ is
  isomorphic to $C^m(G)$, and hence that $C^m(G)$ is ELC-preserved.
\end{proof}

\begin{figure}
 \centering
 {\hfill
 \subfloat[The graph $C^2(s^3)$]{\includegraphics[width=.35\linewidth]{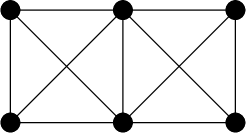}\label{fig:cliqueexample1}}
 \hfill
 \subfloat[The graph $C^2(S_-^2(s^3))$]{\includegraphics[width=.45\linewidth]{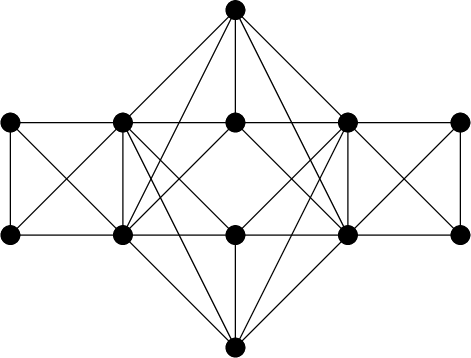}\label{fig:cliqueexample2}}
 \hfill}
 \caption{Examples of clique expansion}\label{fig:cliqueexample}
\end{figure}

Examples of graphs obtained by clique expansion are shown in
Fig.~\ref{fig:cliqueexample}. The output of a clique expansion will
always be a non-bipartite graph, except for the trivial case $C^2(e^1)
= s^2$. However, the input can be a bipartite graph, and hence the
construction can be combined with star expansion to obtain new
ELC-preserved graphs, such as the graph $C^2(S_-^2(s^3))$ of order 12,
shown in Fig.~\ref{fig:cliqueexample2}. Iterating clique expansion on
its own does not produce new graphs, since, trivially, $C^m(c^k) =
c^{mk}$ and $C^{m_2}(C^{m_1}(G)) = C^{m_1m_2}(G)$.

\begin{defn}\label{def:hammingcode}
  Let the graph $h^r$ be an $(r, 2^r-r-1)$-bipartite graph on $n =
  2^r-1$ vertices. To obtain $h^r$, let one partition, $U$, consist of
  $r$ vertices, and the other partition, $W$, be divided into $r-1$
  disjoint subsets, $W_i$, for $2 \le i \le r$, where $W_i$ contains
  $\binom{r}{i}$ vertices. Let each vertex in $W_i$ be connected to
  $i$ vertices in $U$, such that $N_a \ne N_b$ for all $a,b \in W$.
\end{defn}

\begin{thm}\label{thm:hammingcode}
  The graph $h^r$, for $r \ge 3$, is ELC-preserved and corresponds to
  the $[2^r-1, 2^r-r-1, 3]$ \emph{Hamming code}.
\end{thm}
\begin{proof}
  From the construction of the graph $h^r$, we see that it corresponds
  to a code with parity check matrix $(I \mid P)$, where the columns
  are all non-zero vectors from $\mathbb{F}_2^r$, which is the parity
  check matrix of a Hamming code~\cite{handbook}. We know from
  Theorem~\ref{thm:elccode} that any ELC operation on $h^r$ must give
  a graph that corresponds to an equivalent code. Since the distance
  of the code is greater than two, all columns of the parity check
  matrix must be distinct. It follows that all parity check matrices
  of equivalent codes must contain all non-zero vectors from
  $\mathbb{F}_2^r$, in some order. Hence the corresponding graphs are
  isomorphic, and $h^r$ must be ELC-preserved.
\end{proof}

A graph is \emph{even} if all its vertices have even degree, and
\emph{odd} if all its vertices have odd degree. (Connected even graphs
are also known as \emph{Eulerian graphs}.) An odd graph must have even
order, and is always the complement of an even graph. Odd graphs have
been shown to correspond to \emph{Type II} self-dual additive codes
over $\mathbb{F}_4$~\cite{classlc}.

\begin{lem}\label{lem:antieuler}
  Let $G=(V,E)$ be an odd graph. After performing any LC or ELC
  operation on $G$, we obtain a graph $G'$ which is also odd.
\end{lem}
\begin{proof}
  Let $v \in V$ and $w \in N_v$. LC on $v$ transforms $N_w$ into $N_w'
  = (N_w \cup N_v)\setminus(N_w \cap N_v)\setminus\{w\}$, where
  $\left|N_w'\right| = \left|N_w\right| + \left|N_v\right| -
  (2\left|N_w \cap N_v\right| + 1)$. Since $G$ is odd,
  $\left|N_w\right|$ and $\left|N_v\right|$ must be odd. We then see
  that $\left|N_w'\right|$ is the sum of three odd numbers, and must
  therefore be odd. The same argument holds for all neighbors of $v$,
  so $G * v$ is odd. That ELC also preserves oddness then follows from
  Definition~\ref{prop:triplelc}.
\end{proof}

\begin{defn}\label{def:exthammingcode}
  Let the graph $h^r_e$ be an $(r+1, 2^r-r-1)$-bipartite graph on $n =
  2^r$ vertices. To obtain $h^r_e$, first construct $h^r$, as in
  Definition~\ref{def:hammingcode}, and then add a new vertex which is
  connected by edges to all existing vertices of even degree.
\end{defn}

\begin{thm}\label{thm:exthammingcode}
  The graph $h^r_e$, for $r \ge 3$, is ELC-preserved and corresponds
  to the $[2^r, 2^r-r-1, 4]$ \emph{extended Hamming code}.
\end{thm}
\begin{proof}
  $h^r_e$ must be bipartite, since all vertices of $h^r$ in the
  partition of size $r$ have degree $\sum_{i=2}^r \binom{r}{i}
  \frac{i}{r} = 2^{r-1}-1$, which is odd.  The new vertex added to
  $h^r$ also has odd degree, since the number of vertices in $h^r$ of
  even degree is $\sum_{i=1}^{\lfloor\frac{r}{2}\rfloor} \binom{r}{2i}
  = 2^{r-1}-1$. Hence $h^r_e$ is odd. It follows from the construction
  that $h^r_e$ corresponds to a code with parity check matrix $(I \mid
  P)$, where the columns are all odd weight vectors from
  $\mathbb{F}_2^{r+1}$, which is the parity check matrix of an
  extended Hamming code~\cite{handbook}. We know from
  Theorem~\ref{thm:elccode} that any ELC operation on $h^r_e$ must
  give a graph that corresponds to an equivalent code. Since the
  distance of the code is greater than two, all columns of the parity
  check matrix must be distinct.  The graph $h^r_e$ is odd, and must
  remain so after ELC, according to Lemma~\ref{lem:antieuler}. It
  follows that all parity check matrices of equivalent codes must
  contain all odd weight vectors from $\mathbb{F}_2^{r+1}$, in some
  order. Hence the corresponding graphs are isomorphic, and $h^r_e$
  must be ELC-preserved.
\end{proof}

For $n=7$, we obtain from Theorem~\ref{thm:hammingcode} the bipartite
ELC-preserved graph $h^3$, shown in Fig.~\ref{fig:subh}, corresponding
to the Hamming code of length 7. This is an important graph, as it
forms the basis for the general constructions given by
Theorems~\ref{hexapand} and \ref{chexapand}. The graph $h^3_e$ is
shown in Fig.~\ref{fig:h3e}.

\begin{figure}
\centering
\subfloat[The graph $h^3$]
{\hspace{5pt}\includegraphics[height=80pt]{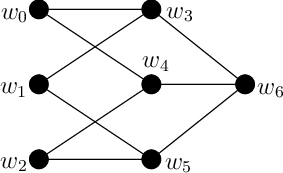}\hspace{5pt}\label{fig:subh}}
 \quad
\subfloat[The graph $h^3_e$]
{\hspace{5pt}\includegraphics[height=80pt]{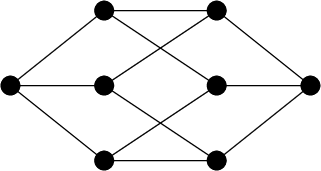}\hspace{5pt}\label{fig:h3e}}
\caption{ELC preserved graphs from Hamming codes}
\end{figure}

\begin{defn}\label{def:hexapand}
  Given a graph $G=(V,E)$ on $k$ vertices, let the \emph{Hamming
    expansion} $H(G)$ be a graph on $n=7k$ vertices constructed as
  follows.  For all vertices $v_i \in V$, $0 \le i < k$, we replace
  $v_i$ by the subgraph $h_i$ with vertices $\{w_{7i}, \ldots,
  w_{7i+6}\}$ and edges $\{\{w_{7i}, w_{7i+3}\}$, $ \{w_{7i},
  w_{7i+4}\}$, $ \{w_{7i+1}, w_{7i+3}\}$, $ \{w_{7i+1}, w_{7i+5}\}$, $
  \{w_{7i+2}, w_{7i+4}\}$, $ \{w_{7i+2}, w_{7i+5}\}$, $ \{w_{7i+3},
  w_{7i+6}\}$, $ \{w_{7i+4}, w_{7i+6}\}$, $ \{w_{7i+5},
  w_{7i+6}\}\}$. (Note that $h_i$ is a specific labeling of the graph
  $h^3$. The labeled graph $h_0$ is depicted in Fig.~\ref{fig:subh}.)
  If $\{v_i, v_j\} \in E$, we connect each of the vertices $w_{7i}$,
  $w_{7i+1}$, and $w_{7i+2}$ to all the vertices $w_{7j}$, $w_{7j+1}$,
  and $w_{7j+2}$. (Note that this differs from the graph substitution
  in Definition~\ref{def:subst}.)  As an example, consider the graph
  $H(s^2)$ shown in Fig.~\ref{fig:hexample1}.
\end{defn}

\begin{thm}[Hamming expansion]\label{hexapand}
  The graph $H(G)$ is ELC-preserved if $G$ is ELC-preserved.
\end{thm}
\begin{proof}
  Let $a=w_6$, $b=w_3$, and $c=w_0$. If $k>1$, let $d=w_7$, and assume
  (without loss of generality) that there is an edge $\{v_0, v_1\} \in
  E$.  Due to the symmetry of $H(G)$ and ELC-preservation of $G$, we
  only need to consider ELC on the three edges $\{a,b\}$, $\{b,c\}$,
  and $\{c,d\}$ to prove the ELC-preservation of $H(G)$.  That $h_0$
  is ELC-preserved, and hence that $\{a,b\}$ preserves $H(G)$ is
  easily verified by hand.  We then consider the edge $\{b,c\}$.  Note
  that $N_b^c = \{a, c'=w_1\}$, where $c'$ has exactly the same
  neighbors as $c$ outside $h_0$, and $a$ has no common neighbors with
  $c$ outside $h_0$.  Since we know that the subgraph $h_0$ is
  ELC-preserved, the effect of ELC on $\{b,c\}$ is simply to swap $a$
  and $c'$.  The edge $\{c,d\}$ corresponds to the edge $\{v_0, v_1\}
  \in E$.  In the graph $G$, let $A = N_{v_0}^{v_1}$, $B =
  N_{v_1}^{v_0}$, and $C = N_{v_0} \cap N_{v_1}$.  In the graph
  $H(G)$, $c$ is connected to three copies of $A$, $d$ is connected to
  three copies of $B$, and both $c$ and $d$ are connected to three
  copies of $C$.  Since ELC on $\{v_0, v_1\}$ preserves $G$, toggling
  pairs between these multiplied neighborhoods must preserve $H(G)$,
  as in Theorem~\ref{cexapand}.  There are only eight remaining
  vertices to consider: $c$ is connected to $D = \{w_3, w_4\}$ and $E
  = \{w_8, w_9\}$, and $d$ is connected to $F = \{w_{10}, w_{11}\}$
  and $G = \{w_1, w_2\}$.  The vertices in $D$ have no neighbors
  outside $h_0$, and the vertices in $F$ have no neighbors outside
  $h_1$. The vertices in $E$ share the same neighbors as $d$ outside
  $h_1$, and the vertices in $G$ share the same neighbors as $c$
  outside $h_0$.  The effect of ELC on $\{c,d\}$ is to swap $D$ with
  $E$ and $F$ with $G$. Hence $H(G)$ must be preserved, except for the
  local structure of $h_0$ and $h_1$, which it remains to check.  ELC
  on $\{c,d\}$ has the effect of toggling pairs between $D$ and $G$
  and between $E$ and $F$.  Finally we swap $u$ and $v$. The result is
  that the structure of $h_0$ and $h_1$ is preserved, as illustrated
  in Fig.~\ref{fig:hexample1} and Fig.~\ref{fig:hexample2}. It follows
  that $H(G)$ is ELC-preserved.
\end{proof}

\begin{figure}
\centering
\includegraphics[height=100pt]{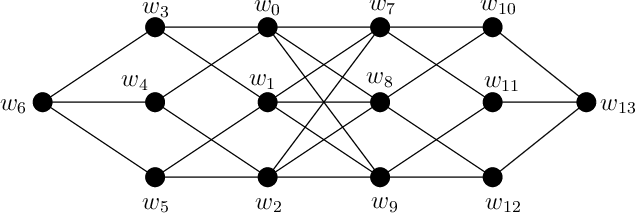}
\caption{The graph $H(s^2)$}\label{fig:hexample1}
\end{figure}

\begin{figure}
\centering
\includegraphics[height=100pt]{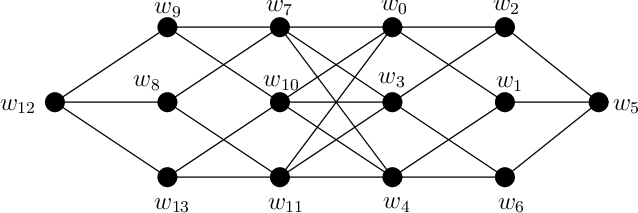}
\caption{The graph $H(s^2)^{(w_0,w_7)}$}\label{fig:hexample2}
\end{figure}

\begin{thm}[Hamming clique expansion]\label{chexapand}
  For $k\ge 1$ and $m \ge 1$, we obtain an ELC-preserved graph $H_k^m$
  on $n=7k+m$ vertices by taking the union of $G = H(c^k)$ and $K =
  c^m$. We add edges from each vertex in $K$ to all the $3k$ vertices
  in $G$ labeled (as in Theorem~\ref{hexapand}) $w_{7i}$, $w_{7i+1}$,
  and $w_{7i+2}$, for $0 \le i < k$. (Note that $H_1^1 = h^3_e$.) As
  an example, the graph $H_3^3$ is shown in Fig.~\ref{fig:hexpex}.
\end{thm}
\begin{proof}
  Without loss of generality, let $a=w_6$, $b=w_3$, $c=w_0$, $d=w_7$,
  and let $e$ and $f$ be two distinct vertices in $K$.  (For $k=1$,
  ignore $d$, and for $m=1$, ignore $f$.)  Due to the symmetry of
  $H_k^m$, we only need to consider ELC on the five edges $\{a,b\}$,
  $\{b,c\}$, $\{c,d\}$, $\{c,e\}$, and $\{e,f\}$ to prove the
  ELC-preservation of $H_k^m$.  The proof for $\{a,b\}$, $\{b,c\}$,
  and $\{c,d\}$ are the same as in Theorem~\ref{hexapand}.  (The proof
  still works with $K=c^m$ added to $N_c$ and $N_d$.)  The edge
  $\{e,f\}$ is trivial, since $N_e = N_f$.  It only remains to show
  that ELC on $\{c,e\}$ preserves $H_k^m$.  Observe that $N_c^e =
  \{w_3, w_4\}$ and $N_e^c = \{w_1, w_2\}$.  All other neighbors of
  $c$ and $e$ are in $N_c \cap N_e$, since the underlying graph of
  $G=H(c^k)$ is a complete graph.  Furthermore, $w_1$ and $w_2$ are
  connected to all vertices in $N_c \cap N_e$, and $w_3$ and $w_4$ are
  not connected to any vertex in $N_c \cap N_e$.  The effect of ELC is
  to swap the vertices in $N_c^e$ with the vertices in $N_e^c$.  $h_0$
  is preserved as before.  It follows that $H_k^m$ is ELC-preserved.
\end{proof}

\begin{figure}
\centering
\includegraphics[height=175pt]{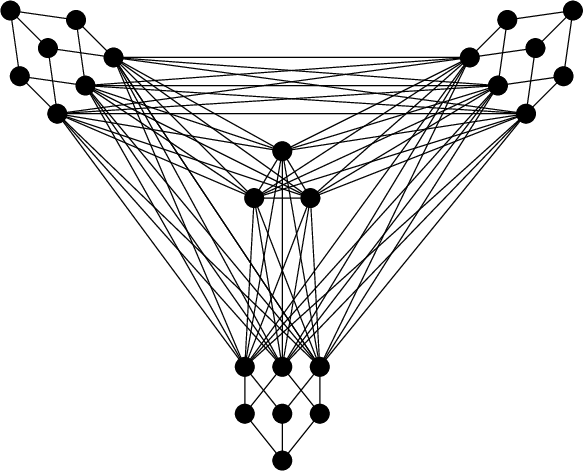}
\caption{The graph $H_3^3$}\label{fig:hexpex}
\end{figure}

\begin{prop}
  $H(G)$ is bipartite when $G=(V,E)$ is bipartite. $H_k^m$ is
  bipartite only in the trivial case where $k=m=1$.
\end{prop}
\begin{proof}
  Let $V = \{v_0, \ldots v_{k-1}\}$. In $H(G)$, each $v_i$ is replaced
  by a bipartite subgraph, $h_i$, and edges are added between these
  subgraphs, such that the induced subgraph on $\{w_{7i},$ $w_{7i+1},$
  $w_{7i+2},$ $w_{7j},$ $w_{7j+1},$ $w_{7j+2}\}$ in $H(G)$ is a
  complete bipartite graph if there is an edge $\{v_i, v_j\} \in E$
  and an empty graph otherwise.  It follows that $H(G)$ is bipartite
  whenever $G$ is bipartite. (The trivial case $H(e^1) = h^3$ is
  clearly also bipartite.)  $H_k^m$ is clearly non-bipartite if $k>2$
  or $m>2$, since it contains a 3-clique.  It is easily checked that
  for the remaining cases, only $H_1^1 = h^3_e$ is bipartite.
\end{proof}

\section{Classification}\label{sec:class}

Tables~\ref{tab:class1} and \ref{tab:class2} show how all bipartite
ELC-preserved graphs of order $n \le 16$, and all non-bipartite
ELC-preserved graphs of order $n \le 12$ arise from the constructions
described in the previous section.

\begin{table}
\centering
\caption{Classification of bipartite ELC-preserved graphs}\label{tab:class1}
\begin{tabular}{rl}
\toprule
$n$ & \\
\midrule
2  & $s^{2}$ \\
3  & $s^{3}$ \\
4  & $s^{4}$  \\
5  & $s^{5}$  \\
6  & $s^{6}$, $S_-^2(s^3)$  \\
7  & $s^{7}$, $h^3$  \\
8  & $s^{8}$, $S_-^2(s^4)$, $h^3_e$  \\
9  & $s^{9}$, $S_-^3(s^3)$  \\
10 & $s^{10}$, $S_-^2(s^5)$  \\
11 & $s^{11}$  \\
12 & $s^{12}$, $S_-^2(s^6)$, $S_-^3(s^4)$, $S_-^4(s^3)$, $S_-^2(S_-^2(s^3))$  \\
13 & $s^{13}$  \\
14 & $s^{14}$, $S_-^2(s^7)$, $S_-^2(h^3)$, $S_+^2(h^3)$, $H(s^2)$ \\
15 & $s^{15}$, $S_-^3(s^5)$, $S_-^5(s^3)$, $h^4$ \\
16 & $s^{16}$, $S_-^2(s^8)$, $S_-^4(s^4)$, $S_-^2(S_-^2(s^4))$, $S^2(h^3_e)$, $h^4_e$ \\
\bottomrule
\end{tabular}
\end{table}

\begin{table}
\centering
\caption{Classification of non-bipartite ELC-preserved graphs}\label{tab:class2}
\begin{tabular}{rl}
\toprule
$n$ & \\
\midrule
3  & $c^3$ \\
4  & $c^4$ \\
5  & $c^5$ \\
6  & $c^6$, $C^2(s^3)$ \\
7  & $c^7$ \\
8  & $c^8$, $C^2(s^4)$ \\
9  & $c^9$, $C^3(s^3)$, $H_1^2$ \\
10 & $c^{10}$, $C^2(s^5)$, $H_1^3$ \\
11 & $c^{11}$, $H_1^4$ \\
12 & $c^{12}$, $C^2(s^6)$, $C^3(s^4)$, $C^4(s^3)$, $C^2(S_-^2(s^3))$, $H_1^5$\\
\bottomrule
\end{tabular}
\end{table}

We observe that certain pairs of ELC-preserved graphs are
LC-equivalent. It is easy to verify that $c^n$ and $s^n$ form a
complete LC orbit, for all $n\ge 3$.  The following theorem explains
all the remaining pairs of LC-equivalent ELC-preserved graphs for $n
\le 12$, namely $\{S_-^2(s^4), C^2(s^4)\}$, $\{S_-^2(s^6),
C^2(s^6)\}$, $\{S_-^3(s^4), C^3(s^4)\}$, and $\{S_-^2(S_-^2(s^3)),
C^2(S_-^2(s^3))\}$. (Note that all these pairs of graphs are part of
larger LC orbits whose other members are not ELC-preserved.)

\begin{thm}
  Let $G = (U \cup W,E)$ be a $(r,n-r)$-bipartite graph with
  partitions $U=\{u_1, \ldots, u_r\}$ and $W=\{w_1, \ldots,
  w_{n-r}\}$. Let $S^m(G)$ be the graph where the vertices in $U$ are
  substituted with $e^m$. If all vertices in $U$ have odd degree, and
  all pairs of vertices from $U$ have an even number of (or zero)
  common neighbors, then $C^m(G) = S^m(G)*w_1* \cdots *w_{n-r}$, i.e.,
  we can get from $S^m(G)$ to $C^m(G)$ by performing LC on all
  vertices in $W$. (The order of the LC operations is not important.)
\end{thm}
\begin{proof}
  Consider performing LC on a vertex $w_i$ in $S^m(G)$. This vertex
  will be connected to the set $X$ of $m-1$ pendant vertices, and to
  $km$ other vertices, where $k$ is the degree of $w_i$ in $G$. Let
  $u$ be a neighbor of $w_i$ in $G$, and let $Y$ be the set of $m$
  vertices that $u$ is replaced with in $S^m(G)$.  The subgraph
  induced on $Y$ is $e^m$. After LC on $w_i$, the induced subgraph on
  $Y$ will be $c^m$.  Moreover, the induced subgraph on $X \cup
  \{w_i\}$ will also be $c^m$, and all vertices in $Y$ will be
  connected to all vertices of $X \cup \{w_i\}$. Subsequent LC on
  another vertex $w_j$, where $w_j$ is also connected to $u$ in $G$,
  will change the subgraph induced on $Y$ back to $e^m$. To ensure
  that the induced subgraph on $Y$ is $c^m$ in the final graph, we
  must require $u$ to have odd degree in $G$.  If $w_i$ is also
  connected to another vertex $u'$ in $G$, which is replaced by $Y'$
  in $S^m(G)$, LC on $w_i$ will connect all vertices in $Y$ to all
  vertices in $Y'$. Since we require that $u$ and $u'$ share an even
  number of neighbors, none of these edges will remain in the final
  graph.  With these considerations, it follows that after performing
  LC on all vertices in $W$, we obtain a graph where every vertex of
  $G$ is substituted by $c^m$, which is the definition of $C^m(G)$.
\end{proof}

New non-bipartite ELC-preserved graphs, $h^r_*$, of order $n = 2^r$
for $r \ge 4$, can be obtained from the following theorem, by applying
the given LC operations to ELC-preserved bipartite graphs
corresponding to extended Hamming codes, $h^r_e$.  The smallest
example of this, $h^4_*$, is shown in Fig.~\ref{fig:hstar2}.  (Note
that $h^3_* = h^3_e$. For $r \ge 4$, $h^r_*$ is a non-bipartite
ELC-preserved graph that cannot be obtained from any of our other
constructions.)

\begin{figure}
 \centering
 \subfloat[The Graph $h^4_e$]
 {\hspace{5pt}\includegraphics[width=.32\linewidth]{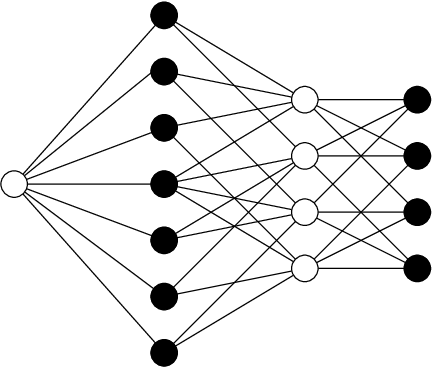}\hspace{5pt}\label{fig:hstar1}}
 \quad
 \subfloat[The Graph $h^4_*$]
 {\hspace{5pt}\includegraphics[width=.52\linewidth]{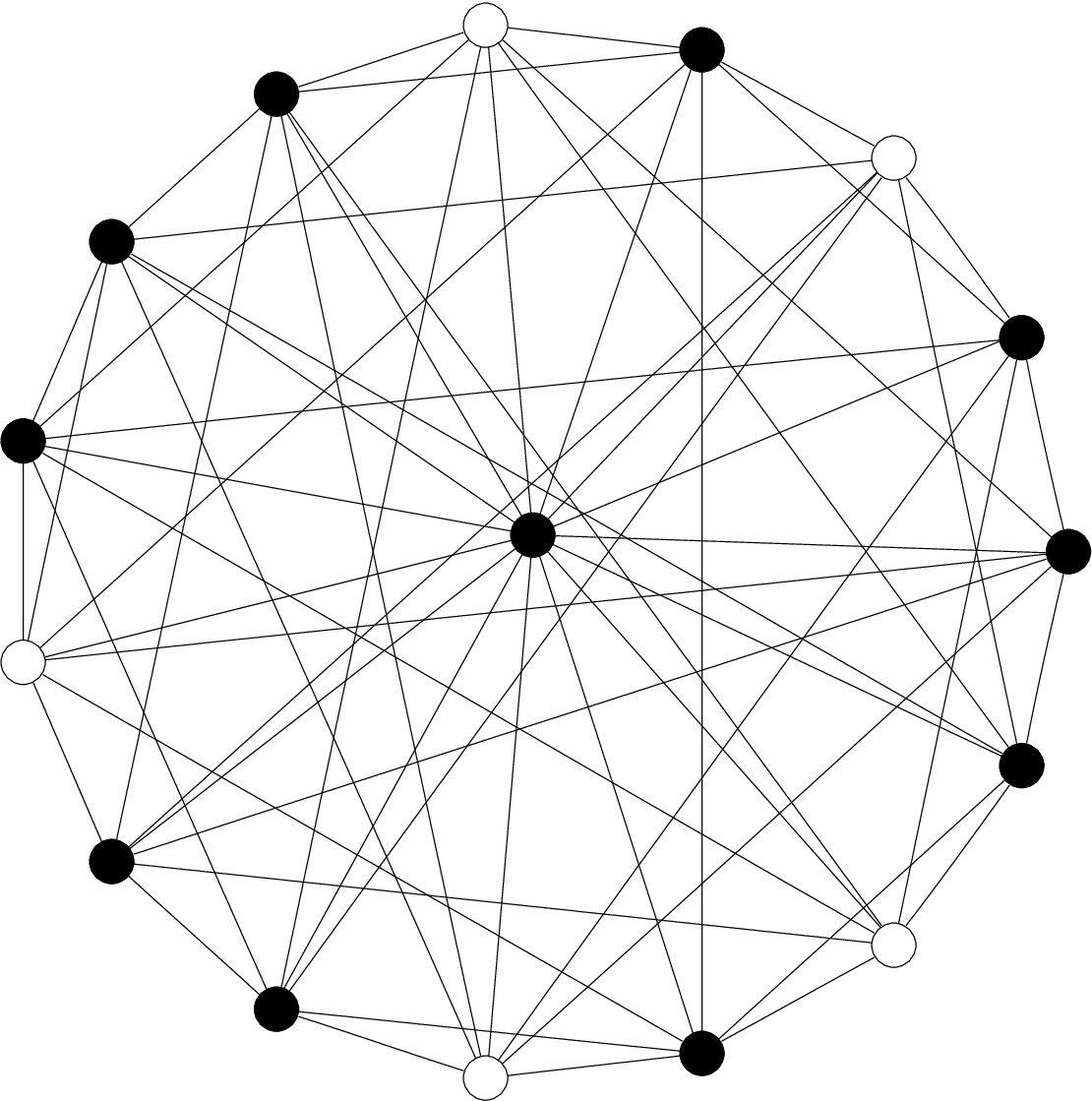}\hspace{5pt}\label{fig:hstar2}}
 \caption{Example of new ELC-preserved graph obtained by LC}\label{fig:hstar}
\end{figure}

\begin{thm}\label{thm:lchamming}
  Given the bipartite ELC-preserved graph $h^r_e$, defined in
  Definition~\ref{def:exthammingcode}, then $2^r-r-1$ LC operations applied, 
  in any order, one to each vertex in the partition of size $2^r-r-1$, 
  preserve the graph, while $r+1$ LC operations applied, in any order, 
  one to each vertex in the partition of size $r+1$ gives an ELC-preserved graph
  $h^r_*$ which is non-bipartite when $r \ge 4$.
\end{thm}
\begin{proof}
  Let $U$ denote the set of vertices in the partition of size $r+1$,
  and $W$ denote the set of vertices in the partition of size
  $2^r-r-1$. After performing LC on all vertices in $W$, two vertices
  $u, v \in U$ will be connected by an edge if and only if $u$ and $v$
  have an odd number of common neighbors in $W$. To show that LC on
  all vertices in $W$ preserves $h^r_e$, we must show that all pairs
  of vertices from $U$ have an even number of common neighbors. Let
  $u_e$ be the extension vertex that was added to $h^r$ to form
  $h^r_e$, as described in Definition~\ref{def:exthammingcode}, and let
  $u_i$ and $u_j$ be two other vertices in $U$. The number of
  neighbors common between $u_e$ and $u_i$ is
  $\sum_{i=1}^{\lfloor\frac{r}{2}\rfloor} \binom{r}{2i} \frac{2i}{r} =
  2^{r-2}$. The number of neighbors common between $u_i$ and $u_j$ is
  $\sum_{i=2}^{r} \binom{r-2}{i-2} = 2^{r-2}$.

  We will now show that LC on all vertices in $U$ transforms $h^r_e$
  into the ELC-preserved graph $h^r_*$. The adjacency matrix of
  $h^r_e$ can be written $\Gamma = \begin{pmatrix}\boldsymbol{0}_{r
      \times r} & P \\ P^\text{T} & \boldsymbol{0}_{(n-r) \times
      (n-r)}\end{pmatrix}$, where $(I \mid P)$ is the parity check
  matrix of $\mathcal{C}$, an extended Hamming code.  LC on a vertex
  $u \in U$ can be implemented on $\Gamma$ by adding row $u$ to all
  rows in $N_u$ and then changing the diagonal elements
  $\Gamma_{v,v}$, for all $v \in N_u$, from 1 to 0.  After performing
  LC on all vertices in $U$, the adjacency matrix of $h^r_*$ is $M
  = \begin{pmatrix}\boldsymbol{0} & P \\ P^\text{T} &
    X\end{pmatrix}$. Since each vertex in $W$ has an odd number of
  neighbors in $U$, each row of $X$ is the linear combination of an
  odd number of rows from $P$, except that all diagonal elements of
  $X$ have been changed from 1 to 0. Moreover, the non-zero coordinates
  of row $i$ of $P^T$ indicate which rows of $P$ were added to form
  row $i$ of $X$. It follows that the rows of the matrix
  $\begin{pmatrix}I & P \\ P^\text{T} & X+I\end{pmatrix}$ are the
  $2^r$ codewords of $\mathcal{C}^\perp$ formed by taking all linear
  combinations of an odd number of rows from $(I \mid P)$, since $(I
  \mid P)$ contains all odd weight columns from $\mathbb{F}_2^{r+1}$.
  After performing ELC on an edge $\{u,v\}$ in $h^r_*$, where $u \in
  U$ and $v \in W$, and then swapping vertices $u$ and $v$, we obtain
  an adjacency matrix $M' = \begin{pmatrix}\boldsymbol{0} & P' \\
    {P'}^\text{T} & X'\end{pmatrix}$.  After ELC on an edge $\{u,v\}$
  where $u,v \in W$, the vertices in $U$ will no longer be an
  independent set, but by permuting vertices from $U$ with vertices
  from $N_u$ or $N_v$, we can obtain the form $M'$.
  We need to show that the rows of $M'+I$ are $2^r$ codewords of a
  code equivalent to $\mathcal{C}^\perp$ formed by taking linear
  combinations of an odd number of rows from $(I \mid P')$.  Since,
  according to Theorem~\ref{thm:exthammingcode}, the extended Hamming
  code only has one parity check matrix, up to column permutations,
  this implies that $h^r_*$ is ELC-preserved.  ELC on $\{u,v\}$ is the
  same as LC on $u$, followed by LC on $v$, followed by LC on $u$
  again. We have seen that LC corresponds to row additions and
  flipping diagonal elements.  We only need to show that all diagonal
  elements of $M$ are flipped from 1 to 0 an even number of times to
  ensure that all rows of $M'+I$ are the codewords described above. If
  we swap vertices $u$ and $v$ after performing ELC, it follows from
  the definition of ELC that rows $u$ and $v$ of $M'$ must be the same
  as in $M$. As for the other rows, LC on $u$ flips $M_{i,i}$ for $i
  \in N_u \setminus \{v\}$ , LC on $v$ then flips $M_{i,i}$ for $i \in
  (N_v \cup N_u) \setminus (N_v \cap N_u)$, and finally, LC on $u$
  flips $M_{i,i}$ for $i \in N_v \setminus \{v\}$. In total, this
  means that for each $i \in N_u \cup N_v \setminus \{u,v\}$, the
  diagonal element $M_{i,i}$ has been flipped from 1 to 0 two times.

  The graph $h^r_*$ is non-bipartite if there is at least one pair of
  vertices from $W$ with an odd number of common neighbors in $U$. For
  $r \ge 4$, there must be a pair of vertices from $W_2 \subset W$,
  the set of $\binom{r}{2}$ vertices of degree $2$ in $h^r$, with no
  common neighbors in $h^r$ and hence one common neighbor, i.e. the
  extension vertex, in $h^r_e$.
\end{proof}

\section{ELC-preserved Codes}\label{sec:coding}

As we have already shown, the graph $h^3$ corresponds to the $[7,4,3]$
\emph{Hamming code}, and its dual $[7,3,4]$ \emph{simplex code}. The
graph $h^3_e$ corresponds to the self-dual $[8,4,4]$ extended Hamming
code. The star graph $s^n$ corresponds to the $[n,1,n]$
\emph{repetition code}, and its dual $[n,n-1,2]$ \emph{parity check
  code}. We can obtain larger ELC-preserved bipartite graphs using
Hamming expansion or star expansion, and the parameters of the
corresponding codes are given by the following theorems.

\begin{thm}\label{thm:hexpcodes}
   For a connected ELC-preserved $(r, k-r)$-bipartite graph $G$ on $k
   \ge 2$ vertices, the graph $H(G)$ corresponds to a $[7k, 3k+r, 4]$
   code $\mathcal{C}$, and to the dual $[7k, 4k-r, 4]$ code
   $\mathcal{C}^\perp$.
\end{thm}
\begin{proof}
  From the construction of $H(G)$, we get that $\mathcal{C}$ must have
  length $n=7k$. The codes $\mathcal{C}$ and $\mathcal{C}^\perp$ have
  dimension $3k+r$ and $4k-r$, respectively, since $H(G)$ has
  partitions of size $3k+r$ and $4k-r$ when $G$ has partitions of size
  $r$ and $k-r$. That both $\mathcal{C}$ and $\mathcal{C}^\perp$ have
  minimum distance 4 follows from the fact that the minimum vertex
  degree in both partitions of $H(G)$ is 3. This is verified by
  observing that the subgraph $h_0$, shown in Fig.~\ref{fig:subh}, has
  one vertex $w_6$ of degree 3, and three vertices $w_3$, $w_4$, and
  $w_5$ of degree 3, belonging to different partitions. Moreover, the
  degrees of $w_0$, $w_1$, and $w_2$ must be at least 5, since $G$ is
  connected.
\end{proof}

\begin{thm}
  Let $G$ be a connected ELC-preserved $(r, k-r)$-bipartite graph on
  $k \ge 2$ vertices and assume, without loss of generality, that $r
  \le k-r$. Let $G$ correspond to a $[k, r, d]$ code and its dual $[k,
  k-r, d']$ code. Then $S_+^m(G)$ corresponds to an $[mk, r, md]$ code
  and its dual $[mk, mk-r, 2]$ code. $S_-^m(G)$ corresponds to an
  $[mk, k-r, md']$ code and its dual $[mk, mk-k-r, 2]$ code.
\end{thm}
\begin{proof}
  From the construction of $S^m(G)$, we get that all the codes must
  have length $n=mk$. In $G$, the minimum vertex degree in the
  partition of size $r$ must be $d-1$, and the minimum vertex degree
  in the other partition must be $d'-1$. In $S_+^m(G)$, $k-r$ vertices
  of $G$ have been substituted by $e^m$ and $m$ pendants have been
  added to the other $r$ vertices. Hence, $S_+^m(G)$ must contain a
  partition of size $r$ with minimum vertex degree $md-1$, since the
  vertex of degree $d-1$ in $G$ is now connected to $d-1$ copies of
  $e^m$ plus $m-1$ pendants. The other partition of $S_+^m(G)$ has
  size $mk-r$, and contains pendants, i.e., vertices of degree one.
  By similar argument, $S_-^m(G)$ has a partition of size $k-r$ with
  minimum vertex degree $md'-1$ and a partition of size $mk-k-r$ with
  minimum vertex degree one. The dimensions and minimum distances of
  the corresponding codes follow.
\end{proof}

We observe that the ELC-preserved graphs $h^3_e$ and $H(s^2)$
correspond to $[8,4,4]$ and $[14,7,4]$ self-dual codes. A natural
question to ask is whether there are other ELC-preserved self-dual
codes. All self-dual binary codes of length $n \le 34$ have been
classified by Bilous and van~Rees~\cite{bilous,bilous2}. A database
containing one representative from each equivalence class of codes
with $n \le 32$ and $d \ge 4$, and one representative from each
equivalence class of codes with $n = 34$ and $d \ge 6$ is available
on-line at \url{http://www.cs.umanitoba.ca/~umbilou1/}. We have
generated the ELC orbits of all the corresponding bipartite graphs,
and found that $h^3_e$ and $H(s^2)$ are the only ELC-preserved graphs,
as shown in Table~\ref{tab:selfdual}. However, as the following
theorem shows, we can construct ELC-preserved self-dual codes with $n
\ge 56$ by iterated Hamming expansion of $h^3_e$ and $H(s^2)$.

\begin{table}
\centering
\caption{ELC orbit size of graphs corresponding to self-dual codes}
\label{tab:selfdual}
\begin{tabular}{rrrrr}
\toprule
$n$ & $d$ & Codes & ELC-preserved & Size two ELC orbits \\
\midrule
8  & $\ge 4$ & 1    & 1 & - \\
10 & $\ge 4$ & -    & - & - \\
12 & $\ge 4$ & 1    & - & 1 \\
14 & $\ge 4$ & 1    & 1 & - \\
16 & $\ge 4$ & 2    & - & 1 \\
18 & $\ge 4$ & 2    & - & - \\
20 & $\ge 4$ & 6    & - & 1 \\
22 & $\ge 4$ & 8    & - & - \\
24 & $\ge 4$ & 26   & - & 2 \\
26 & $\ge 4$ & 45   & - & - \\
28 & $\ge 4$ & 148  & - & 1 \\
30 & $\ge 4$ & 457  & - & - \\
32 & $\ge 4$ & 2523 & - & 2 \\
34 & $\ge 6$ & 938  & - & - \\
\bottomrule
\end{tabular}
\end{table}

\begin{thm}
  Let $H^r(G) = H(\cdots H(G))$ denote the $r$-fold Hamming expansion
  of $G$. Then for $r \ge 1$, $H^r(h^3_e)$ corresponds to an
  ELC-preserved self-dual $[8 \cdot 7^r, 4 \cdot 7^r, 4]$ code, and
  $H^{r+1}(s^2)$ corresponds to an ELC-preserved self-dual $[2 \cdot
  7^{r+1}, 7^{r+1}, 4]$ code.
\end{thm}
\begin{proof}
  The parameters of the codes follows from
  Theorem~\ref{thm:hexpcodes}. It remains to show that they are
  self-dual. A code with generator matrix $(I \mid P)$ is self-dual if
  the same code is also generated by $(P^\text{T} \mid I)$, i.e., if
  $P^{-1} = P^\text{T}$. The codes associated with both $h^3_e$ and
  $H(s^2)$ have the property that $P=P^\text{T}$, and Hamming
  expansion must preserve this symmetry since it has the same effect
  on both partitions of the graph. In general, $P=P^\text{T}$ only
  implies that a code is isodual, but we can prove a stronger property
  in this case. Note that $P$ corresponding to $H(G)$ will have full
  rank when $P$ corresponding to $G$ has full rank, since we know that
  $P$ corresponding to $H(s^2)$, which is the Hamming expansion of the
  induced subgraph on any pair of vertices connected by an edge in
  $G$, has full rank. Since an ELC-preserved code only has one
  generator matrix, up to column permutations, and the inverse of a
  symmetric matrix is symmetric, we must have that $P^{-1}(I \mid P) =
  (P \mid I)$. Hence the code is self-dual.
\end{proof}

\section{Orbits of Size Two}\label{sec:sizetwo}

ELC-preserved codes with good properties could have practical
applications in iterative decoding~\cite{castle,isit,wbelc,isit2010}.
However, there seem to be extremely few such codes, and, except for
the perfect Hamming codes, graphs arising from the constructions in
Section~\ref{sec:const} correspond to $[n,k,d]$ codes with either low
minimum distance $d$ or low rate $\frac{k}{n}$, compared to the best
known codes of the same length.  Iterative decoding with ELC also
works for graphs with larger ELC orbits, such as \emph{quadratic
  residue (QR) codes}~\cite{isit}, and has performance close to that
of iterative permutation decoding~\cite{halford} for graphs with small
ELC orbits, such as the \emph{extended Golay code}~\cite{isit}.  The
self-dual $[24,12,8]$ extended Golay code corresponds to a bipartite
graph with an ELC orbit of size two.  We have also found a $[15,5,7]$
\emph{Bose-Chaudhuri-Hocquenghem (BCH) code} with an ELC-orbit of size
two~\cite{wbelc}, but observed that larger QR and BCH codes have much
larger orbits. As a generalization of ELC-preserved graphs, we now
briefly consider graphs with ELC orbits of size two. The number of
size two orbits are listed in Table~\ref{tab:enumerate2}.  

We have also counted LC orbits of size two. (Note that the only
connected graphs of order up to 12 which have LC orbits of size one
are the trivial graphs of order one and two, and it remains an open
problem to prove that these two graphs are the only graphs with LC
orbits of size one.) Clearly there is an LC orbit $\{s^n, c^n\}$ for
all $n \ge 3$.  The only other size two LC orbit we find for $n \le
12$ is comprised of the two graphs of order six depicted in
Fig.~\ref{fig:hexaorbit} (These two graphs correspond to the self-dual
\emph{Hexacode} over $\mathbb{F}_4$~\cite{classlc}.)

\begin{table}
\centering
\caption{Number of orbits of size two}
\label{tab:enumerate2}
\begin{tabular}{rrrr}
\toprule
$n$ & Bipartite ELC & Non-bipartite ELC & LC \\
\midrule
3  &           - &        - &          1 \\
4  &           1 &        1 &          1 \\
5  &           2 &        3 &          1 \\
6  &           4 &        9 &          2 \\
7  &           6 &       10 &          1 \\
8  &           9 &       21 &          1 \\
9  &          12 &       22 &          1 \\
10 &          22 &       43 &          1 \\
11 &          22 &       41 &          1 \\
12 &          33 &       91 &          1 \\
13 &          35 &          &            \\
14 &          53 &          &            \\
15 &          48 &          &            \\
\bottomrule
\end{tabular}
\end{table}

\begin{figure}
  \centering
  {\includegraphics[width=.30\linewidth]{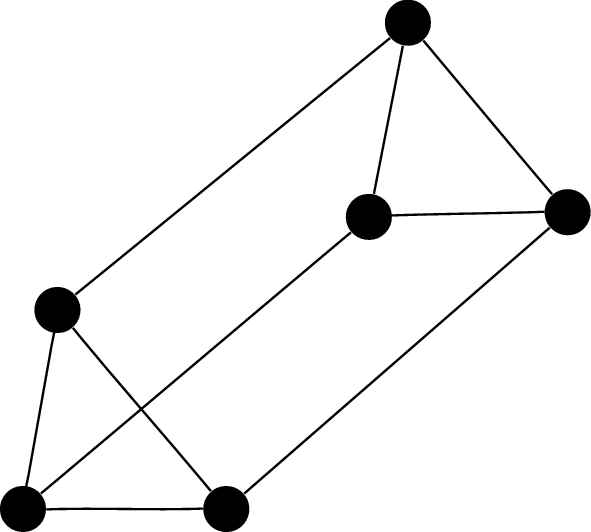}
  \quad
  \includegraphics[width=.30\linewidth]{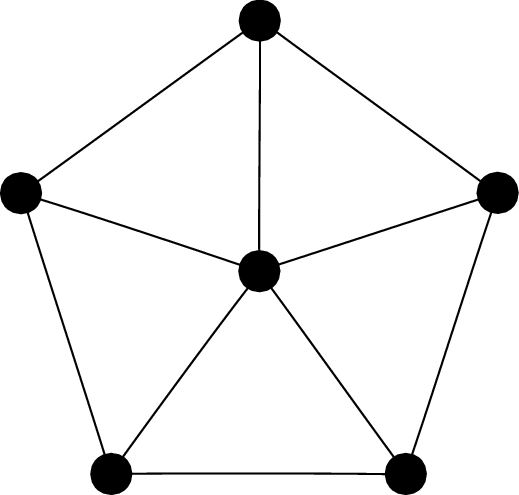}}
  \caption{LC orbit of size two}\label{fig:hexaorbit}
\end{figure}

We have also looked at the ELC orbits corresponding to self-dual codes
of length $n \le 34$, as seen in Table~\ref{tab:selfdual}. Except for
the $[24,12,8]$ extended Golay code and a $[32,16,4]$ code, the
remaining self-dual codes in this table with ELC orbits of size two,
all with minimum distance four, can be constructed by the following
theorem. It remains an open problem to devise a general construction
for self-dual codes with ELC orbits of size two and minimum distance
greater than four.

\begin{thm}
  Let $G$ be a $(2m,2m)$-bipartite graph on $4m$ vertices, where $m
  \ge 3$.  Let the vertices in one partition be labeled $v_1, v_2,
  \ldots, v_{2m}$, and the vertices in the other partition be labeled
  $w_1, w_2, \ldots, w_{2m}$. Let there be an edge $\{v_i,w_j\}$
  whenever $i \ne j$. Then $G$ has an ELC orbit of size two and
  corresponds to a self-dual $[4m, 2m, 4]$ code $\mathcal{C}$.
\end{thm}
\begin{proof}
  The code $\mathcal{C}$ has generator matrix $(I \mid P)$ where $P$
  is circulant with first row $(01\cdots1)$. It can be verified that
  $P^{-1} = P = P^\text{T}$ when $P$ is of this form with even
  dimensions. Hence $P^{-1}(I \mid P) =(P^\text{T} \mid I)$ and
  $\mathcal{C}$ is self-dual.  (An $(m,m)$-bipartite graph constructed
  as above for odd $m \ge 7$ would still have an ELC orbit of size two
  but would correspond to a non-self-dual $[2m, m, 4]$ code.) Note
  that $m=1$ and $m=2$ must be excluded, since they produce the
  ELC-preserved graphs $s^2$ and $h^3_e$, respectively.

  Due to the symmetry of $G$ we only need to consider ELC on one edge
  $\{v_i, w_j\}$. This will take us to a graph $G'$ where the
  neighborhoods of $v_i$, $v_j$, $w_i$, $w_j$ are unchanged, but where
  $N_{v_k} = \{v_i, v_j, w_k\}$ and $N_{w_k} = \{w_i, w_j, v_k\}$, for
  all $k \ne i,j$. We need to consider ELC on three types of edges in
  $G'$. ELC on $\{v_i, w_j\}$ or $\{v_j, w_i\}$ will take us back to
  $G$. ELC on an edge $\{v_k, w_k\}$ will preserve $G'$, since it
  simply removes edges $\{v_i, w_j\}$ and $\{v_j, w_i\}$ and adds
  edges $\{v_i, w_i\}$ and $\{v_j, w_j\}$, thus in effect swapping
  vertices $v_i$ and $v_j$. Finally, ELC on an edge $\{v_i, w_k\}$
  also preserves $G'$, since it swaps the roles of vertices $v_j$ and
  $v_k$. (ELC on $\{v_k, w_i\}$ similarly swaps $w_j$ and $w_k$.) This
  can be seen by noting that $w_k$ has neighbors $v_j$ and $v_k$, with
  $v_k$ being connected to $w_j$ in $N_{v_i}^{w_k}$ and $v_j$ being
  connected to all vertices in $N_{v_i}^{w_k}$ except $w_j$. Hence
  these relations are reversed after complementation. Furthermore,
  $N_{v_k} \setminus N_{v_i}^{w_k} = N_{v_j} \setminus N_{v_i}^{w_k} =
  \{w_k, w_i\}$, so isomorphism is preserved. We have shown that the
  ELC orbit of $G$ has size two. Since the minimum vertex degree over
  the ELC orbit is 3, the minimum distance of $\mathcal{C}$ is 4.
\end{proof}

\section{Conclusions}\label{sec:conc}

We have introduced ELC-preserved graphs as a new class of graphs,
found all ELC-preserved graphs of order up to 12 and all ELC-preserved
bipartite graphs of order up to 16, and shown how all these graphs
arise from general constructions. It remains an open problem to prove
that all ELC-preserved graphs arise from these constructions, or give
an example to the contrary. We therefore pose the following question.

\begin{prob}
Is a connected ELC-preserved graph of order $n$ always either $s^n$,
where $n$ is prime, $H_k^m$, where $n = 7k + m$, $h^r$, where $n = 2^r
- 1$, $h^r_e$ or $h^r_*$, where $n = 2^r$, or can it be obtained as
$S_+^m(G)$, $S_-^m(G)$, or $C^m(G)$, where $G$ is an ELC-preserved
graph of order $\frac{n}{m}$, or $H(G)$, where $G$ is an ELC-preserved
graph of order $\frac{n}{7}$?
\end{prob}

Note that not all star graphs and complete graphs are primitive
ELC-preserved graphs, since most of them can be obtained as
follows. From the graph $e^1$, we can obtain all $c^n = C^n(e^1)$.
From $s^2 = C^2(e^1)$, we obtain all $s^n = S^{\frac{n}{2}}(s^2)$
where $n$ is even.  More generally, for $n=pq$ a composite number,
$s^n = S_+^p(s^{\frac{n}{q}}) = S_+^q(s^{\frac{n}{p}})$, so only $s^p$
with $p$ an odd prime is a primitive ELC-preserved graph.

\begin{prob}
Enumerate or classify ELC-preserved graphs of order $n > 12$ and
ELC-preserved bipartite graphs of order $n > 16$.
\end{prob}
 
Our classification used a previous complete classification of ELC
orbits~\cite{classelc}, and a graph extension technique to obtain all
bipartite ELC-preserved graphs of order 16. Perhaps the complexity of
classification could be reduced by further exploiting restrictions on
the structure of ELC-preserved graphs.

ELC-preserved graphs are an interesting new class of graphs from a
theoretical point of view. As discussed in Section~\ref{sec:intro}, LC
and ELC orbits of graphs show up in many different fields of research,
and ELC-preserved graphs may also be of interest in these contexts. We
have seen that one possible use for bipartite ELC-preserved graphs is
in iterative decoding of error-correcting codes.  Hamming codes are
perfect, but for this application we would like codes with rate
$\frac{k}{n} \approx \frac{1}{2}$.  Such ELC-preserved codes obtained
from our constructions do not have minimum distance that can compete
with the best known codes of similar length, except for the optimal
$[8,4,4]$ code ($h^3_e$), for which iterative decoding has been
simulated with good results~\cite{castle}, and the optimal $[14,7,4]$
code ($H(s^2)$).  Longer codes obtained from Hamming expansion will
always have minimum distance~4, as shown in
Theorem~\ref{thm:hexpcodes}.  Codes that have a negligible number of
low weight codewords can still have good decoding performance, but the
number of weight~4 codewords in these codes grows linearly with the
length, since the number of degree~3 vertices in the corresponding
graphs does so, and hence the codes are not well suited for this
application. It is therefore interesting to consider ELC orbits of
size two, one of them corresponding the extended Golay code of
length~24, for which iterative decoding with ELC has been simulated
with good results~\cite{isit}.  For codes of higher length, however,
this criteria is probably also too restrictive. Graphs with ELC orbits
of bounded size could be more suitable for this application, and would
be interesting to study from a graph theoretical point of view.  For
some graphs, ELC on certain edges will preserve the graph, while ELC
on other edges may not.  Iterative decoding where only ELC on the
subset of edges that preserve the graph are allowed has been
studied~\cite{castle}.  Graphs where ELC on certain edges preserve the
number of edges in the graph, or keep the number of edges within a
given bound, have also been considered in iterative
decoding~\cite{wbelc}.  ELC-preserved graphs are clearly a subclass of
the graphs where all ELC orbit members have the same number of edges.
This class of graphs, and other possible generalizations of
ELC-preserved graphs, would be interesting to study further.

\paragraph*{Acknowledgements}
This research was supported by the Research Council of Norway.  The
authors would like to thank the anonymous reviewers for providing
useful suggestions and corrections that improved the quality of the
manuscript.

\end{document}